\documentclass[11pt]{article}
\usepackage{yub}
\usepackage{songmei}
\usepackage{subfigure}

\newcommand{\rgrad}{{\mathop{\rm grad\,}}}
\newcommand{\rhess}{{\mathop{\rm Hess\,}}}
\newcommand{\proj}{{\mc{P}}}

\def\smoothf{{\beta_f}}
\def\smoothF{{\beta_F}}
\def\smoothP{{\beta_P}}
\def\smoothr{{\beta_E}}

\def\lipf{{L_f}}
\def\lipF{{L_F}}
\def\minsvF{\gamma_F}

\def\rcond{\kappa_R}
\def\cd{{C_d}}

\def\cra{C_{r,1}}
\def\crb{C_{r,2}}
\def\crat{\wt{C}_{r,1}}
\def\crbt{\wt{C}_{r,2}}
\def\cb{{C_b}}
\def\rhof{{\rho_f}}
\def\rhoF{{\rho_F}}

\def\cM{{\mathcal M}}
\def\id{{\mathsf I}}
\def\sT{{\mathsf T}}
\def\proj{{\mathsf P}}
\def\rhess{{\rm Hess}}
\def\D{{\mathsf D}}
\def\rgrad{{\rm grad}}
\def\smoothD{{\beta_D}}

\def\Gf{{G_f}}
\def\Lbfeps{{\underline{f}}}

\def\hessmin{{\lambda_{\min}}}
\def\hessmax{{\lambda_{\max}}}

\def\shownotes{0}  
\ifnum\shownotes=1
\newcommand{\authnote}[2]{$\ll$\textsf{\footnotesize #1 notes: #2}$\gg$}
\else
\newcommand{\authnote}[2]{}
\fi

\begin{document}
\title{Analysis of Sequential Quadratic Programming \\
  through the Lens of Riemannian Optimization}
\author{Yu~Bai\footnote{Department of Statistics, Stanford
    University.~\texttt{yub@stanford.edu}.}
  \and
  Song~Mei\footnote{Institute for Computational and Mathematical
    Engineering, Stanford University.~\texttt{songmei@stanford.edu}.}}
\maketitle

\begin{abstract}
  We prove that a ``first-order'' Sequential Quadratic Programming
  (SQP) algorithm for equality constrained optimization has local
  linear convergence with rate $(1-1/\rcond)^k$, where $\rcond$ is the
  condition number of the Riemannian Hessian, and global convergence
  with rate $k^{-1/4}$. Our analysis builds on insights from
  Riemannian optimization -- we show that the SQP and Riemannian
  gradient methods have nearly identical behavior near the constraint
  manifold, which could be of broader interest for understanding
  constrained optimization.
\end{abstract}


\section{Introduction}

In this paper, we consider the equality-constrained optimization problem
\begin{equation}
  \label{problem:manopt}
  \begin{aligned}
    \minimize_{x \in \R^n} & ~~ f(x), \\
    \subjectto & ~~ x \in \cM = \{x : F(x)=0 \},
  \end{aligned}
\end{equation}
where we assume $f : \R^n \rightarrow \R$ and $F: \R^n \rightarrow
\R^m$ are $C^2$ smooth functions with $m\le n$.

The focus of this paper is on the local and global convergence rate of
``first-order'' methods for~\eqref{problem:manopt}: methods that only
query $\grad f(x)$ at each iteration (but can do whatever they want
with the constraint), e.g. projected gradient descent.
The iteration complexity of first-order unconstrained
optimization has been a foundational result in theoretical machine
learning~\cite{NemirovskiYu83,Bubeck15}, and it would be of interest
to improve our understanding in the constrained case too.




While numerous ``first-order'' methods can solve
problem~\eqref{problem:manopt}~\cite{Bertsekas99,nocedal2006sequential},
we will restrict attention to two types of methods: Riemannian
first-order methods and Sequential Quadratic Programming, which we now
briefly review. When $\mc{M}$ has a manifold structure near $x_\star$,
one could use Riemannian optimization algorithms~\cite{AbsilMaSe09},
whose iterates are maintained on the constraint set $\mc{M}$.
Classical Riemannian algorithms proceed by computing the
Riemannian gradient and then taking a descent step along the geodesics
based on this gradient~\cite{luenberger1972gradient,gabay1982minimizing}.
Later, Riemannian algorithms are simplified
by making use of \textit{retraction}, a mapping from the tangent space
to the manifold that can replace the necessity of computing the exact
geodesics while still maintaining the same convergence
rate. Intuitively, first-order Riemannian methods can be viewed as
variants of projected gradient descent that utilize the manifold
structure more carefully. Analyses of many such Riemannian algorithms
are given in \citep[Section 4]{AbsilMaSe09}.

An alternative approach for solving problem~\eqref{problem:manopt} is
Sequential Quadratic Programming (SQP)~\citep[Section
18]{nocedal2006sequential}.  Each iteration of SQP
solves a quadratic programming problem which minimizes a quadratic
approximation of $f$ on the linearized constraint set $\set{x:F(x_k) +
  \grad F(x_k)(x-x_k)=0}$. 
When the quadratic approximation uses the Hessian of the objective
function, the SQP is equivalent to Newton method solving nonlinear
equations. When the full Hessian is intractable, one can either
approximate the Hessian with BFGS-type updates, or just use some raw
estimate such as a big PSD matrix~\cite{boggs1995sequential}. The
iterates need not be (and are often not) feasible, which makes SQP
particularly appealing when it is intractable to obtain a feasible
start or projection onto the constraint set.


\subsection{Contribution and related work}
In this paper, we consider the following ``first-order'' SQP
algorithm,
which sequentially solves the quadratic program
\begin{equation}
  \label{algorithm:constrained}
  \begin{aligned}
    x_{k+1} = 
    \argmin & ~~  f(x_k) +\<\nabla f(x_k), x-x_k\> +
    \frac{1}{2\eta}\dl x-x_k \dl_2^2 \\
    \subjectto & ~~ F(x_k) + \nabla F(x_k)(x-x_k) = 0,
  \end{aligned}
\end{equation}
where $\eta>0$ is the stepsize. Each iterate only requires
$\set{\grad f, \grad F}$ (hence first-order). This algorithm can be
seen as a cheap prototype SQP (compared with BFGS-type) and is more
suitable than Riemannian methods when the retraction onto the
constrain set is intractable.


We prove that the SQP~\eqref{algorithm:constrained} has local linear
convergence with rate $(1-1/\rcond)^k$ where $\rcond$ is the condition
number of the Riemannian Hessian at $x_\star$
(Theorem~\ref{theorem:convergence}), and global convergence with rate
$k^{-1/4}$ (Theorem~\ref{thm:global_closeness_and_convergence}).
Our work differs from the existing literature in the following ways.

\begin{enumerate}
\item We provide explicit convergence rates which is lacking in prior
  work on SQP. Existing local convergence analysis has focused more on
  the local quadratic convergence of more expensive BFGS-type
  SQPs~\cite{boggs1995sequential,nocedal2006sequential}, whereas
  global convergence results are mostly
  asymptotic~\cite{boggs1995sequential,solodov2009global}.
\item We observe and make explicit the fact that the SQP iterates stay
  ``quadratically close'' to the manifold when initialized near it
  (though potentially far from $x_\star$) -- see
  Figure~\ref{fig:quadratic} for an illustration.
  Such an observation
  connects first-order SQP to Riemannian gradient methods and allows
  us to borrow insights from Riemannian optimization to analyze the SQP.
\item We provide new analysis plans for SQP, based on the fact that
  SQP iterates quickly becomes nearly identical to Riemannian gradient
  steps once it gets near the constraint set.
  Our local analysis builds on a new potential function
  \begin{equation*}
    \| \proj_{x_\star}(x_k - x_\star) \|_2^2 + \sigma \|
    \proj_{x_\star}^\perp (x_k - x_\star) \|_2
  \end{equation*}
  for some $\sigma > 0$ (see Section~\ref{sec:geometry} for definition
  of the projections), as opposed to the traditionally used exact
  penalty functions. Our global analysis
  constructs descent lemmas similar to those in Riemannian gradient
  methods with additional second-order error terms. These results can
  be of broader interest for understanding constrained optimization.



\end{enumerate}

\begin{figure}[h!]
  \begin{center}
    \includegraphics[width=0.4\textwidth]{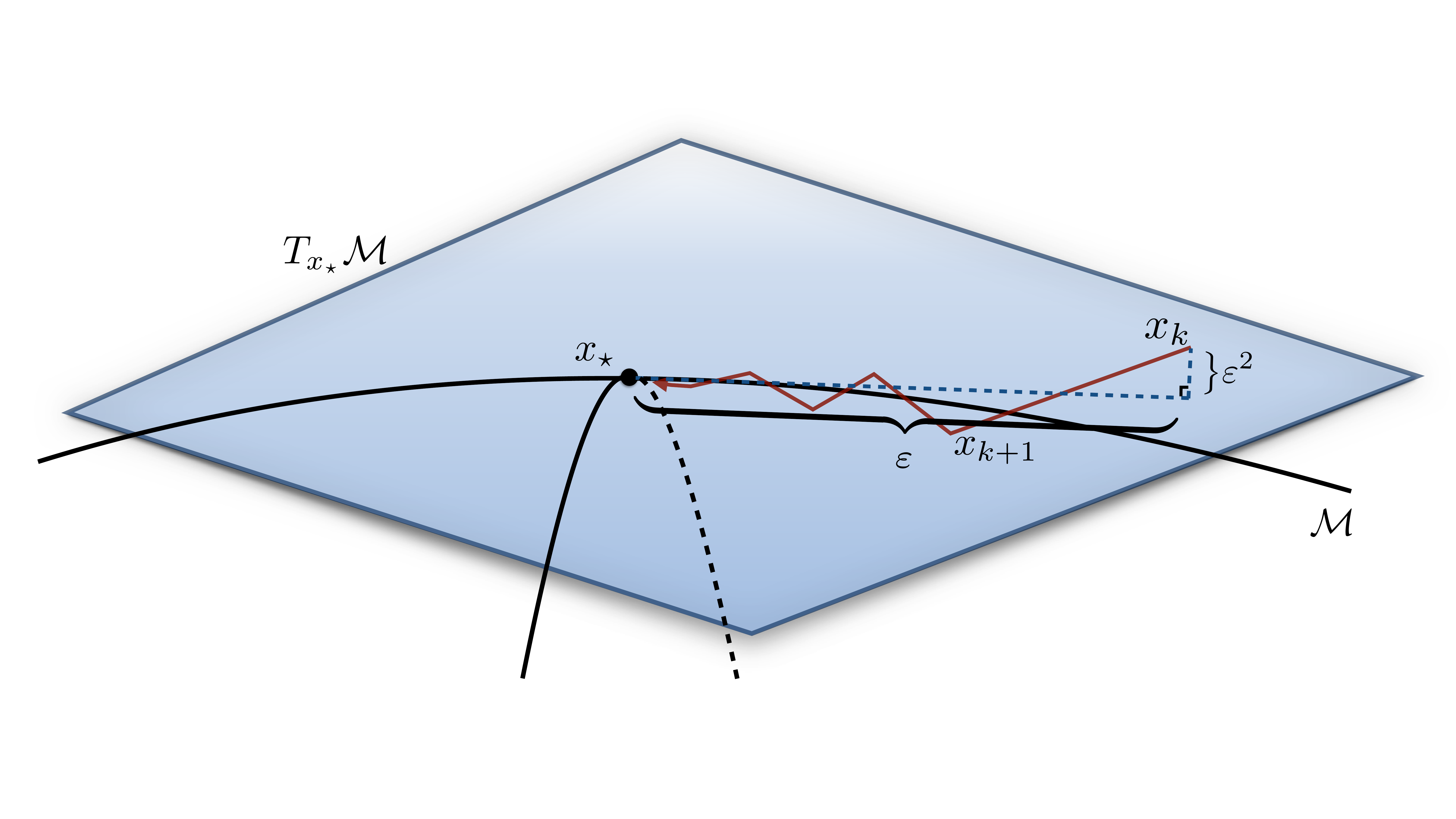}
  \end{center}
  \caption{Illustration for SQP iterates}\label{fig:quadratic}
\end{figure}

\paragraph{Related work}
The convergence rate of many first-order algorithms on
problem~\eqref{problem:manopt}
are shown to achieve
local linear convergence with rate $(1 - 1/\rcond)$, which is termed as
the \textit{canonical rate} in~\cite{luenberger2008linear}.
Riemannian algorithms that achieve the canonical rate include geodesic
gradient projection~\citep[Theorem 2]{luenberger1972gradient},
geodesic steepest descent~\citep[Theorem 4.4]{gabay1982minimizing},
Riemannian gradient descent~\citep[Theorem 4.5.6]{AbsilMaSe09}. The
canonical rate is also achieved by the modified Newton method on the
quadratic penalty function~\citep[Section 15.7]{luenberger2008linear},
which resembles an SQP method in spirit.

Though analyses are well-established for both the Riemannian and the
SQP approach (even by the same author in~\cite{gabay1982minimizing}
for the Riemannian approach and~\cite{gabay1982reduced} for the SQP
approach), the connection between them did not receive much
attention until the past 10 years. Such connection
is re-emphasized in~\cite{absil2009all}: the authors
pointed out that the feasibly-projected sequential quadratic
programming (FP-SQP) method in \cite{wright2004feasible} gives the
same update as the Riemannian Newton update. \citet{MishraSe16} used 
this connection to provide a framework for 
selecting a preconditioning metric for Riemannian optimization, in particular 
when the Riemannian structure is sought on a quotient manifold. 
However, these connections
are established for second-order methods between the Riemannian and the
SQP approaches, and such connection for first-order methods are
not---we believe---explicitly pointed out yet.

\paragraph{Paper organization}
The rest of this paper is organized as the following. In
Section~\ref{sec:preliminaries}, we state our assumptions and give
preliminaries on the Riemannian geometry on and off the manifold
$\mc{M}$.  We present our local and global convergence result in
Section~\ref{sec:main} and~\ref{sec:global}, and prove them in
Section~\ref{sec:proof} and~\ref{sec:proof-global}. 
We give an example in Section~\ref{section:example} and perform
numerical experiments in Section~\ref{sec:exp}.


\subsection{Notation}
For a matrix $A \in \R^{m \times n}$, we denote
$A^{\dagger} \in \R^{n \times m}$ as its Moore-Penrose inverse,
$A^\sT \in \R^{n \times m}$ as its transpose, and $A^{\dagger \sT}$ as
the transpose of its Moore-Penrose inverse. As $n \ge m$ and $\rank(A) = m$, we have
$A^\dagger = A^\sT(A A^\sT)^{-1}$. We denote $\sigma_{\min}(A)$ to be 
the least singular value of matrix $A$. For a $k$'th order tensor
$T \in \R^{n_1 \times n_2 \times \cdots \times n_k}$, and $k$ vectors
$u_1 \in \R^{n_1}, \ldots, u_k \in \R^{n_k}$, we denote
$T[u_1, \ldots, u_k] = \sum_{i_1, \ldots, i_k} T_{i_1\cdots i_k} u_{1,
  i_1} \cdots u_{k, i_k}$ as tensor-vectors multiplication. The
operator norm of tensor $T$ is defined as
$\dl T \dl_{\op} = \sup_{\dl u_1 \dl_2 = 1, \ldots, \dl u_k \dl_2 = 1}
T[u_1, \ldots, u_k]$. 

For a scaler-valued function $f: \R^n \rightarrow \R$, we write its
gradient at $x \in \R^n$ as a column vector $\nabla f(x) \in \R^n$. We
write its Hessian at $x \in \R^n$ as a matrix
$\nabla^2 f(x) \in \R^{n \times n}$, and its third order derivative at
$x$ as a third order tensor
$\nabla^3 f(x) \in \R^{n \times n \times n}$. For a vector-valued
function $F: \R^n \rightarrow \R^m$, its Jacobian matrix at
$x \in \R^n$ is an $m \times n$ matrix
$\nabla F(x) \in \R^{m \times n}$, and its Hessian matrix at $x$ as a
third order tensor $\nabla^2 F(x) \in \R^{m \times n \times n}$.

\section{Preliminaries}\label{sec:preliminaries}
\subsection{Assumptions}\label{sec:assumptions}
Let $x_\star$ be a local minimizer of
problem~\eqref{problem:manopt}.
Throughout the rest of this paper, we make the following assumptions
on problem~\eqref{problem:manopt}. In particular, all these
assumptions are local, meaning that they only depend on the properties
of $f$ and $F$ in $\ball(x_\star,\delta)$ for some $\delta>0$.
\begin{assumption}[Smoothness]
  \label{assumption:smoothness}
  Within $\ball(x_\star, \delta)$, the functions $f$ and $F$ are $C^2$
  with local Lipschitz constants
  $\lipf, \lipF$, Lipschitz gradients with constants $\smoothf$,
  $\smoothF$, and Lipschitz Hessians with constants $\rhof$, $\rhoF$. 
\end{assumption}
\begin{assumption}[Manifold structure and constraint qualification]
  \label{assumption:cq}
  The set $\mc{M}$ is a $m$-dimensional smooth submanifold of $\R^n$.
  Further, $\inf_{x \in \ball(x_\star, \delta)}
  \sigma_{\min}(\nabla F(x))\ge\minsvF$ for some constant
  $\minsvF>0$. 
\end{assumption}
Smoothness and constraint qualification together implies that the
constraints $F(x)$ are well-conditioned and
problem~\eqref{problem:manopt} is $C^2$ near $x_\star$. In particular,
we can define a matrix $\rhess f(x_\star)$ via the formula
\begin{equation}
  \begin{aligned}
    & \quad \rhess f(x_\star) = \proj_{x_\star} \nabla^2 f(x_\star)
    \proj_{x_\star} - \sum_{i=1}^m [ \nabla F(x_\star)^{\dagger \sT} \nabla
    f(x_\star)]_i \cdot  \proj_{x_\star} \nabla^2 F_i(x_\star)
    \proj_{x_\star}, 
  \end{aligned}
\end{equation}
where
\begin{equation}
\begin{aligned}
\proj_{x_\star} = \id_n - \nabla F(x_\star)^\dagger \nabla F(x_\star). 
\end{aligned}
\end{equation}
We can see later that $\rhess f(x_\star)$ is the matrix representation of the Riemannian Hessian of function $f$ on $\cM$ at $x_\star$. 
\begin{assumption}[Eigenvalues of the Riemannian Hessian]
  \label{assumption:quadratic-growth}
  Define
  \begin{equation*}
    \begin{aligned}
      \lambda_{\max} = \sup\{ \< u, \rhess f(x_\star) u\> : \dl u
      \dl_2 = 1, \nabla F(x_\star) u = 0 \},\\ 
      \lambda_{\min} = \inf\{ \< u, \rhess f(x_\star) u\>: \dl u \dl_2
      = 1, \nabla F(x_\star) u = 0 \}.\\ 
    \end{aligned}
  \end{equation*}
  We assume $0 < \hessmin \le \hessmax < \infty$. We call $\rcond =
  \hessmax/\hessmin$ the condition number of $\rhess f(x_\star)$. 
\end{assumption}

\subsection{Geometry on the manifold $\cM$}
\label{sec:geometry}
Since we assumed that the set $\cM$ is a smooth submanifold of $
\R^{n}$, we endow $\cM$ with the Riemannian geometry induced by the
Euclidean space $\R^{n}$. At any point $x \in \cM$, the tangent space
(viewed as a subspace of $\R^n$) is obtained by taking the
differential of the equality constraints
\begin{equation}
T_{x} \cM= \{  u \in \R^n : \nabla F(x) u = 0  \}.
\end{equation}
Let $\proj_x$ be the orthogonal projection operator from $\R^n$ onto $T_x \cM$. For any $u \in \R^n$, we have
\begin{equation}
\begin{aligned}
\proj_x (u) &= [\id_n - \nabla F(x)^\sT (\nabla F(x) \nabla F(x)^\sT)^{-1} \nabla F(x) ] u. 
\end{aligned}
\end{equation}
Let $\proj_x^{\perp}$ be the orthogonal projection operator from $\R^n$ onto the complement subspace of $T_x \cM$. For any $u \in \R^n$, we have
\begin{equation}
\begin{aligned}
\proj_x^{\perp} (u) &= \nabla F(x)^\sT (\nabla F(x) \nabla F(x)^\sT)^{-1} \nabla F(x) u. 
\end{aligned}
\end{equation}
With a little abuse of notations, we will not distinguish $\proj_x$ and $\proj_x^{\perp}$ with their matrix representations. That is, we also think of $\proj_x, \proj_x^{\perp} \in \R^{n \times n}$ as two matrices. 

We denote $\nabla f(x)$ and $\rgrad f(x)$ 
respectively the Euclidean gradient and the Riemannian gradient of $f$ at $x \in \cM$. 
The Riemannian gradient of $f$ is the projection of the Euclidean gradient onto the tangent space
\begin{equation}
  \begin{aligned}
    \rgrad f(x) = \proj_x (\nabla f(x)) = [\id_n - \nabla F(x)^\sT (\nabla F(x) \nabla F(x)^\sT)^{-1}
    \nabla F(x) ] \nabla f(x).
  \end{aligned}
\end{equation}
Since $x_\star$ is a local minimizer of $f$ on the manifold $\cM$, we have $\rgrad f(x_\star) = 0$.

At $x \in \cM$, let $\nabla^2 f(x)$ and $\rhess f(x)$ be respectively the Euclidean and the Riemannian Hessian of $f$. The Riemannian Hessian is a symmetric operator on the tangent space and is given by projecting the directional derivative of the gradient vector field. That is, for any $u, v \in T_x \cM$, we have (we use $\D$ to denote the directional derivative)
\begin{equation}
  \begin{aligned}
    & \rhess f(x) [u, v] = \< v, \proj_x ( \D \rgrad f(x) [u] )\> \\
    = & \<v, \proj_x \cdot \nabla^2 f(x) u - \proj_x \cdot (\D
    \proj_x^{\perp} [u])\cdot \nabla f\> \\
    = & v^\sT \proj_x \nabla^2 f(x) \proj_x u - \nabla^2 F(x)[\nabla
    F(x)^{\dagger \sT} \nabla f(x), u, v]. 
\end{aligned}
\end{equation}
With a little abuse of notation, we will not distinguish the Hessian operator with its matrix representation. That is
\begin{equation}
  \begin{aligned}
    \rhess f(x)  = \proj_x \nabla^2 f(x) \proj_x - \sum_{i=1}^m [ \nabla
    F(x)^{\dagger \sT} \nabla f(x)]_i \cdot  \proj_x \nabla^2 F_i(x)
    \proj_x.
  \end{aligned}
\end{equation}

\subsection{Geometry off the manifold $\cM$}
We can extend the definition of the matrix representations of the above Riemannian quantities outside the manifold $\cM$. For any $x \in \R^n$, we denote 
\begin{equation}
\begin{aligned}
& \quad \proj_x = \id_n - \nabla F(x)^\sT (\nabla F(x) \nabla F(x)^\sT)^{-1} \nabla F(x),\\
& \quad \proj_x^{\perp} = \nabla F(x)^\sT (\nabla F(x) \nabla F(x)^\sT)^{-1} \nabla F(x),\\
& \quad \rgrad f(x) = \proj_x \grad f(x).
\end{aligned}
\end{equation}

By the constraint qualification assumption (Assumption
\ref{assumption:cq}), $\nabla F(x)\nabla F(x)^\sT$ is invertible and
the quantities above are well defined in $\ball(x_\star, \delta)$. We
call $\rgrad f(x)$ the extended Riemannian gradient of $f$ at $x$,
which extends the Riemannian gradient outside the manifold $\cM$ as
$(f,F)$ are still well-defined there. 

\subsection{Closed-form expression of the SQP iterate}
The above definitions makes it possible to have a concise closed-form
expression for the SQP iterate~\eqref{algorithm:constrained}. Indeed,
as each iterate solves a standard QP, the expression can be obtained
explicitly by writing out the optimality condition. Letting $x_k = x$,
the next iteration $x_{k+1} = x_+$ is given by 
\begin{equation}
  \label{equation:delta}
  \begin{aligned}
    & x_+  = x  -\eta [\id_n - \nabla F(x)^\sT (\nabla F(x) \nabla
    F(x)^\sT )^{-1} \nabla F(x)]\nabla f(x) \\
    & \quad\quad\quad - \nabla F(x)^\sT (\nabla F(x) \nabla
    F(x)^\sT )^{-1} F(x) \\
    & = x - \eta\proj_x \grad f(x) - \nabla F(x)^\sT (\nabla F(x) \nabla
    F(x)^\sT )^{-1} F(x) \\
    & = x - \eta\rgrad f(x) - \grad F(x)^\dagger F(x).
  \end{aligned}
\end{equation}
We will frequently refer to this expression in our proof.

\section{Main results}
\subsection{Local convergence theorem}
\label{sec:main}
Under Assumptions~\ref{assumption:smoothness},~\ref{assumption:cq}
and~\ref{assumption:quadratic-growth}, we show that the SQP
algorithm~\eqref{algorithm:constrained} converges locally linearly
with rate $1-1/\rcond$. The proof can be found in
Section~\ref{sec:proof}.
\begin{theorem}[Local linear convergence of SQP with canonical rate]
  \label{theorem:convergence}
  There exists $\eps > 0$ and a constant $\sigma >0$ such that the
  following holds. Let $x_0 \in \ball(x_\star, \eps)$ and $x_k$ be
  the iterates of Equation (\ref{algorithm:constrained}) with stepsize
  $\eta=1/\lambda_{\max}(\rhess f(x_\star))$. Letting
  \begin{equation*}
    \begin{aligned}
      a_k =& \dl \proj_{x_\star} (x_k-x_\star) \dl_2,\\
      b_k =& \dl \proj_{x_\star}^\perp (x_k-x_\star) \dl_2,
    \end{aligned}
  \end{equation*}
  we have
  \begin{equation*}
    a_{k+1}^2 + \sigma b_{k+1} \le \Big(1 -
      \frac{1}{2\rcond}\Big)^2(a_k^2 + \sigma b_k), 
  \end{equation*}
  where
  $\rcond = \lambda_{\max}(\rhess f(x_\star))/\lambda_{\min}(\rhess
  f(x_\star))$ is the condition number of the Riemannian Hessian of
  function $f$ on the manifold $\cM$ at $x_\star$. Consequently, the
  distance $\dl x_k-x_\star \dl_2^2=a_k^2+b_k^2$ also converges
  linearly: $$\ltwo{x_k-x_\star}^2\le O\left((1 -
    1/(2\rcond))^k\right).$$
\end{theorem}
{\bf Remark}
Theorem~\ref{theorem:convergence} requires choosing the stepsize
$\eta$ according to the maximum eigenvalue of $\rhess f(x_\star)$,
which might not be known in advance. In practice, one could
implement a line-search (for example as in ~\citep[Section
4.2]{AbsilMaSe09} for Riemannian gradient methods) which would
hopefully achieve the same optimal rate.

\subsection{Global convergence}
\label{sec:global}
Let $\cM_\eps = \{ x : \| F(x) \|_2 \le \eps \}$ denote an
$\eps$-neighborhood of the manifold $\cM$. To show global properties
of SQP algorithm, we make the following additional assumptions:
\begin{assumption}[Global assumptions]\label{ass:global}~
\begin{enumerate}
\item There exists $\eps_0 \ge 0$, such that $\sup_{x \in
    \cM_{\eps_0}} \| \nabla f(x) \|_2 \le \Gf$, and $\inf_{x \in
    \cM_{\eps_0}} f(x) \ge \Lbfeps$.
\item The condition in Assumption \ref{assumption:cq} holds in this
  neighborhood $\cM_{\eps_0}$: $\inf_{x \in
    \cM_{\eps_0}}\sigma_{\min}(\nabla F(x))\ge\minsvF$.
\item Some conditions in Assumption \ref{assumption:smoothness} holds
  globally in $\R^n$: the functions $f$ and $F$ are $C^2$, $f$ has
  Lipschitz constant $\lipf$, and $f$ and $F$ have Lipschitz
  gradients with constants $\smoothf$, $\smoothF$.
\end{enumerate}
\end{assumption}

The following theorem establishes the global convergence of SQP
algorithm with a small constant stepsize.
Our convergence guarantee is provided in terms of the norm of the
extended Riemannian gradient. The proof can be found in
Section~\ref{sec:proof-global}.
\begin{theorem}\label{thm:global_closeness_and_convergence}
There exists constants $K_1, K_2 > 0, \eps_\star > 0$, such that for any $\eps \le \eps_\star$, we initialize $x_0 \in \cM_\eps$, and letting step size to be $\eta = \sqrt{\eps/K_1}$, then each iterates will be close to the manifold, i.e., $\{ x_i \}_{i \in \N} \subseteq \cM_\eps$. Moreover, for any $k \in \N$, we have 
\[
\min_{i \in [k]} \| \rgrad f(x_i) \|_2^2 \le K_2 \{ [f(x_0) -  \Lbfeps]/ (k \sqrt \eps) + \sqrt{\eps} \}. 
\]
\end{theorem}

To minimize the bound in the right hand side, one can choose $\eps = O(1/ k)$, and we get the following corollary. 
\begin{corollary}
There exists constants $\{ K_i \}_{i=1}^4$, such that for any $k \ge K_1$, if we take $\eps = K_2/k$, and initialize on the manifold $\cM$ with step size $\eta = K_3/k^{1/2}$, we have 
\[
\min_{i \in [k]} \| \rgrad f(x_i) \|_2 \le K_4/k^{1/4}.
\] 
\end{corollary} {\bf Remark on the stationarity measure} While our
global convergence is measured the extended Riemannian gradient
$\ltwo{\rgrad f(x_i)}$ on the infeasible iterate $x_i$'s, one could
construct nearby feasible points with small Riemannian gradient via a
straightforward perturbation argument.



\section{Proof of Theorem~\ref{theorem:convergence}}\label{sec:proof}
\subsection{Riemannian Taylor expansion}
\label{sec:RTE}
The proof of Theorem~\ref{theorem:convergence} relies on a particular
expansion of the Riemannian gradient off the manifold, which we state
as follows.
\begin{lemma}[First-order expansion of Riemannian gradient]
  \label{lemma:riemannian-taylor-expansion}
  There exist constants $\eps_0>0$ and $\cra, \crb>0$ such that for all
  $x\in\ball(x_\star,\eps_0)$,
  \begin{equation}
    \rgrad f(x) = \rhess f(x_\star) (x - x_\star) +
    r(x),
  \end{equation}
  where the remainder term $r(x)$ satisfies the error bound
  \begin{equation}
    \| r(x) \|_2 \le \cra \| x-x_\star \|_2^2 +
    \crb \| \proj_{x_\star}^\perp(x-x_\star) \|_2.
  \end{equation}
\end{lemma}
Lemma~\ref{lemma:riemannian-taylor-expansion} extends
classical Riemannian Taylor expansion~\citep[Section
7.1]{AbsilMaSe09} to points off the manifold, where the remainder
term contains an additional first-order error. While the error term
is linear in $\| \proj_{x_\star}^\perp(x-x_\star) \|_2$, this
expansion is particularly suitable when $\| \proj_{x_\star}^\perp(x-x_\star) \|_2$ is on the order
of $ \|\proj_{x_\star}(x-x_\star) \|_2^2$, which results in a
quadratic bound on $\ltwo{r(x)}$. 

In the proof of Theorem~\ref{theorem:convergence}, $\rgrad f(x)$ mostly
appears through its squared norm and inner product with
$x-x_\star$. We summarize the expansion of these terms in the
following corollary.
\begin{corollary}
  \label{corollary:rgrad-expansions}
  There exist constants $\eps_0>0$ and $\cra,\crb>0$ such that the
  following hold. For any $\eps\le \eps_0$, $x\in\ball(x,\eps)$,
  \begin{align*}
    & \<\rgrad f(x), x-x_\star\> = \<x-x_\star, \rhess
      f(x_\star)(x-x_\star)\> + R_1, \\
    & \| \rgrad f(x)\|_2^2 = \<x-x_\star, [\rhess
      f(x_\star) ]^2(x-x_\star)\> + R_2,
  \end{align*}
  where
  \begin{align*}
    \max\set{|R_1|, |R_2|}
    \le \eps (\cra\| \proj_{x_\star}(x-x_\star)\|_2^2 +
      \crb \| \proj_{x_\star}^\perp(x-x_\star)\|_2 ).
  \end{align*}
\end{corollary}

\subsection{Proof of Theorem~\ref{theorem:convergence}}
The following perturbation bound (see, e.g.~\cite{Sun01}) for
projections is useful in the proof.
\begin{lemma}
  \label{lemma:projection-perturbation-bound}
  For $x_1,x_2 \in \ball(x_\star, \delta)$, we have
  \begin{equation*}
    \begin{aligned}
      &\opnorm{\proj_{x_1} - \proj_{x_2}} \le
      \smoothP \dl x_1-x_2 \dl_2 \\
      & \dl \proj_{x_1} \proj_{x_2}^\perp \dl_{\op} \le \smoothP  \dl
      x_1-x_2 \dl_2.
    \end{aligned}
  \end{equation*}
  where $\smoothP = 2\smoothF/\minsvF$.
\end{lemma}
We now prove the main theorem.

\paragraph{Step 1: A trivial bound on $\| x_+-x_\star\|_2^2$. }
Consider one iterate of the algorithm $x\to x_+$, whose closed-form
expression is given in~\eqref{equation:delta}. We have $x_+ =x
+\Delta$, where
\begin{equation}
\Delta = -\eta \cdot \rgrad f(x) - \nabla F(x)^\dagger F(x). 
\end{equation}
We would like to relate $\dl x_+ - x_\star \dl_2$ with $\dl x - x_\star \dl_2$. 
Observe that $\rgrad f(x_\star)=0$, we have
\begin{equation*}
  \dl \rgrad f(x) \dl_2 = \dl \rgrad f(x) - \rgrad f(x_\star) \dl_2 \le
  \smoothr \dl x-x_\star \dl_2.
\end{equation*}
Observe that $ F(x_\star) = 0$, we have
\begin{equation*}
\begin{aligned}
  &\quad \dl \nabla F(x)^\dagger F(x) \dl_2 = \dl \nabla F(x)^\dagger
  (F(x) - F(x_\star)) \dl_2 \\
  & \le \dl \nabla F(x)^\dagger \dl_{\op} \dl F(x) - F(x_\star) \dl_2
  \le (\smoothF/\minsvF) \dl x - x_\star \dl_2. 
  \end{aligned}
\end{equation*}
Accordingly, we have
\begin{align*}
  & \quad \dl x_+-x_\star \dl_2  \\
  & \le \dl x-x_\star \dl_2 + \eta \dl
    \rgrad f(x) \dl_2 + \dl \nabla F(x)^\dagger F(x) \dl_2  \\
  & \le [1 + \eta \smoothr + (\smoothF/\minsvF)] \dl x-x_\star \dl_2.
\end{align*}
Hence, for any stepsize $\eta$, letting $\cd= [1 + \eta \smoothr + (\smoothF/\minsvF)]^2$, for $x \in \ball(x_\star, \delta)$, we have
\begin{equation} \label{equation:distance-growth}
  \dl x_+-x_\star \dl_2^2 \le \cd \dl x-x_\star\dl_2^2. 
\end{equation}

\paragraph{Step 2: Analyze normal and tangent distances. }
This is the key step of the proof.
We look into the \emph{normal} direction and the \emph{tangent} direction separately. The intuition
for this process is that the normal part of $x-x_\star$ is a measure of feasibility, and as we will see, converges much more quickly.

Now we look at equation $x_+-x_\star=x-x_\star+\Delta$. Multiplying it by $\proj_x$ and
$\proj_x^\perp$ gives
\begin{align}
  &  \proj_x^\perp (x_+-x_\star) = \proj_x^\perp  (x-x_\star) + \proj_x^\perp \Delta =
    \proj_x^\perp  (x-x_\star) -  \nabla F(x)^\dagger
    F(x), \label{equation:px-perp} \\
  & \proj_x (x_+-x_\star) = \proj_x (x-x_\star) + \proj_x \Delta = \proj_x (x-x_\star) - \eta
    \cdot \rgrad f(x). \label{equation:px}
\end{align}
We now take squared norms on both equalities and bound the
growth. Define the normal and tangent distances as
\begin{equation}
  a = \dl \proj_{x_\star} (x-x_\star)\dl_2,~~~
  b = \dl \proj_{x_\star}^\perp (x-x_\star) \dl_2 ,
\end{equation}
and $(a_+,b_+)$ similarly for $x_+$. Note that the definitions of $a$,
$b$ use the projection at $x_\star$, so they are slightly different
from quantities~\eqref{equation:px-perp} and~\eqref{equation:px}.

From now on, we assume
that $\| x-x_\star\|_2\le\eps_0$, and $\eps_0$ is sufficiently small
such that~\eqref{equation:distance-growth} holds. The requirements on
$\eps_0$ will later be tightened when necessary. 

\subparagraph{The normal direction. }
We have
\[
\begin{aligned}
  & \quad \dl \proj_x^\perp (x_+ - x_\star)\dl_2^2 \dl \proj_x^\perp (x - x_\star) \dl_2^2 + 2 \< x - x_\star,
  \proj_x^\perp \Delta \> + \dl \proj_x^\perp \Delta \dl_2^2 \\ 
  &= \dl \proj_x^\perp (x - x_\star) \dl_2^2 - 2 \< x - x_\star,
  \nabla F(x)^\dagger F(x) \> + \<F(x), (\nabla F(x) \nabla
  F(x)^\sT)^{-1} F(x) \> \\ 
  &= \dl \proj_x^\perp (x - x_\star) \dl_2^2 - 2\big\<\nabla
  F(x)(x-x_\star), (\nabla F(x)\nabla F(x)^\sT)^{-1}(\nabla
  F(x)(x-x_\star)+r(x))\big\> \\
  & \quad\quad+ \big\<\nabla
  F(x)(x-x_\star)+r(x), (\nabla F(x)\nabla F(x)^\sT)^{-1}(\nabla
  F(x)(x-x_\star)+r(x))\big\> \\
  &= \<r(x), (\nabla F(x)\nabla F(x)^\sT)^{-1}r(x)\>, 
\end{aligned}
\]
where $r(x)=F(x) - \nabla F(x)(x-x_\star)$. By the smoothness of function $F$, we have
\[
\dl r(x) \dl_2 \le  \smoothF/2 \cdot \dl x-x_\star \dl_2^2. 
\]
Accordingly, we get
\begin{equation*}
\begin{aligned}
  & \quad \dl \proj_x^\perp(x_+-x_\star) \dl_2^2 \\
  & \le
  \smoothF^2/(4\minsvF^2) \cdot \dl x-x_\star \dl_2^4 =
  \smoothF^2/(4\minsvF^2) \cdot [\dl 
  \proj_{x_\star}^\perp(x-x_\star) \dl_2^2 + \dl
  \proj_{x_\star}(x-x_\star)\dl_2^2]^2\\ 
  & =  \smoothF^2/(4\minsvF^2) \cdot (a^2+b^2)^2.
  \end{aligned}
\end{equation*}
Applying the perturbation bound on projections
(Lemma~\ref{lemma:projection-perturbation-bound}), we get 
\begin{equation}
  \label{equation:bound-normal}
  \begin{aligned}
    &\quad b_+ = \dl \proj_{x_\star}^\perp (x_+ - x_\star) \dl_2 \\
    & \le
    \dl \proj_x^\perp(x_+-x_\star) \dl_2 +
    \dl \proj_{x}-\proj_{x_\star} \dl_{\op} \dl x_+-x_\star \dl_2 \\
    &\le \smoothF/(2\minsvF) \cdot (a^2+b^2) +
    \smoothP \dl x-x_\star \dl_2 \dl x_+-x_\star\dl_2 \\
    &\le \smoothF/(2\minsvF) \cdot (a^2+b^2) +
    \smoothP\cd \cdot \dl x-x_\star\dl_2^2 \\
    &= \underbrace{[ \smoothF/(2\minsvF) + \smoothP\cd ]}_{\cb} \cdot
    (a^2+b^2) = \cb(a^2+b^2).
  \end{aligned}
\end{equation}

\subparagraph{The tangent direction. }
We have
\begin{equation*}
\begin{aligned}
  &\quad \dl\proj_x (x_+-x_\star) \dl_2^2 \\
  &=  \dl \proj_x (x - x_\star) \dl_2^2 + 2 \< x - x_\star, \proj_x
  \Delta \> + \dl \proj_x \Delta \dl_2^2\\
  &= \dl\proj_x (x-x_\star)\dl_2^2 - 2\eta \cdot \<\rgrad f(x),
  x - x_\star\> + \eta^2 \cdot \dl \rgrad f(x)\dl_2^2.
\end{aligned}
\end{equation*}
Applying Lemma~\ref{lemma:projection-perturbation-bound}, we get that
for any vector $v$,
\begin{equation*}
  \begin{aligned}
    \vert \dl \proj_x v \dl_2^2 - \dl \proj_{x_\star} v \dl_2^2 \vert
    = \vert \<v,(\proj_{x_\star}-\proj_x)v\> \vert \le \smoothP \dl
    x-x_\star \dl_2 \cdot \dl v\dl_2^2.
  \end{aligned}
\end{equation*}
Applying this to vectors $x_+-x_\star$ and $x-x_\star$ gives
\begin{equation}
  \label{equation:aplus-bound}
  \begin{aligned}
    a_+^2
    & \le a^2 - 2\eta\<\rgrad f(x),
    x -x_\star\> + \eta^2\dl \rgrad f(x) \dl_2^2  \\
    & \quad + \smoothP(\dl x-x_\star\dl_2^3 +
    \dl x-x_\star\dl_2 \dl x_+ -x_\star\dl_2^2)  \\
    & \le a^2 - 2\eta\<\rgrad f(x),
    x -x_\star\> + \eta^2\dl \rgrad f(x) \dl_2^2  \\
    & \quad + \smoothP(1+\cd)\dl x-x_\star
    \dl_2^3. 
  \end{aligned}
\end{equation}
Applying Corollary~\ref{corollary:rgrad-expansions}, and
note that $\rhess f(x_\star)=\proj_{x_\star}\rhess
f(x_\star)\proj_{x_\star}$ by the property of the Riemannian Hessian,
we get
\begin{align*}
  & \quad a_+^2 \le a^2 - 2\eta\<x-x_\star, \rhess
    f(x_\star)(x-x_\star)\> + 
   \eta^2\<x-x_\star, (\rhess f(x_\star) )^2(x-x_\star)\> \\ 
  & \quad\quad\quad + (-2\eta R_1 + \eta^2 R_2) + \eps_0 C(a^2+b^2) \\
  & \quad\quad = \<\proj_{x_\star}(x-x_\star), (\id -\eta\rhess
    f(x_\star))^2\proj_{x_\star}(x-x_\star)\> + (-2\eta R_1 +
    \eta^2 R_2) + \eps_0 C(a^2+b^2),
\end{align*}
where the remainders are bounded as
\begin{equation*}
  \begin{aligned}
    \max\set{|R_1|, |R_2|} &\le
    \eps_0\big(\cra\| \proj_{x_\star}(x-x_\star)\|_2^2 +
    \crb\| \proj_{x_\star}^\perp(x-x_\star)\|_2 \big) \\
    & = \eps_0(\cra a^2 + \crb b).
  \end{aligned}
\end{equation*}
Choosing the stepsize as $\eta=1/\lambda_{\max}(\rhess f(x_\star))$,
we have $\id -\eta\rhess f(x_\star) \preceq (1-1/\rcond) \id$. For this choice of
$\eta$, using the above bound for $R_1,R_2$, we get that there exists
some constant $C_1,C_2$ such that
\begin{equation}
  \label{equation:bound-tangent}
  a_+^2 \le
  \Big(1-\frac{1}{\rcond}\Big)^2a^2 + \eps_0(C_1 a^2 + C_2 b).
\end{equation}

\subparagraph{Putting together. }
Let $\sigma > 0$ be a constant to be
determined. Looking at the quantity $a_+^2+ \sigma b_+$, by the
bounds~\eqref{equation:bound-normal}
and~\eqref{equation:bound-tangent} we have
\begin{align*}
  & \quad a_+^2 + \sigma b_+ \\
  & \le \Big(1-\frac{1}{\rcond}\Big)^2a^2
    + \eps_0(C_1a^2+C_2b) + \sigma\cb(a^2+b^2) \\
  & \le \Big( \Big(1-\frac{1}{\rcond}\Big)^2 + \eps_0 C_1 +
    \sigma\cb \Big)a^2 + \eps_0(C_2+ \sigma\cb)b.
\end{align*}
The last inequality is by rearranging the terms and by the assumption that $b = \dl \proj_{x_\star}^\perp (x-x_\star) \dl_2 \le \eps_0$. 
Then, we would like to choose $\sigma$ and $\eps_0$ sufficiently small such that
\begin{align}
\Big(1-\frac{1}{\rcond}\Big)^2 + \eps_0 C_1 +
    \sigma\cb \le& \Big( 1 - \frac{1}{2\rcond} \Big)^2, \label{eqn:thm1_margin1}\\
\eps_0(C_2+ \sigma\cb) \le& \sigma.\label{eqn:thm1_margin2}
\end{align}
This can be obtained by first choosing $\sigma$ small to satisfy Eq. (\ref{eqn:thm1_margin1}) leaving a small margin for the choice of $\eps_0$, then choosing $\eps_0$ small to satisfy both Eq. (\ref{eqn:thm1_margin1}) and (\ref{eqn:thm1_margin2}). With this choice of $\sigma$ and $\eps_0$, we have
\begin{equation*}
  a_+^2 + \sigma b_+ \le \Big( 1 - \frac{1}{2\rcond} \Big)^2(a^2 + 
  \sigma b). 
\end{equation*}

\paragraph{Step 3: Connect the entire iteration path. }
Consider iterates $x_k$ of the first-order
algorithm~\eqref{algorithm:constrained}. Initializing $x^0$
sufficiently close to $x_\star$, the descent on $a_k^2+\sigma b_k$
ensures $a_k^2+b_k^2\le\eps_0^2$. Thus we chain this analysis on
$(x, x_+)=(x_k,x_{k+1})$ and get that $a_k^2+\sigma b_k$ converges
linearly with rate $1-1/(2\rcond)$.  This in turn implies the bound
$a_k^2+b_k^2\le O((1-1/(2\rcond))^k)$ by observing that
$a_k^2+b_k^2\le C(a_k^2+\sigma b_k)$ for some constant $C$.

\section{Proof of Theorem~\ref{thm:global_closeness_and_convergence}}
\label{sec:proof-global}

The proof of Theorem~\ref{thm:global_closeness_and_convergence} is
directly implied by the following two lemmas.
Lemma~\ref{lem:global_closeness} ensures that, as we initialize close
to the manifold and choosing step size reasonably small, the iterates
will be automatically close to the
manifold. Lemma~\ref{lem:global_convergence} shows that, as the
iterates are close to the manifold at each step, the value of function
$f$ can decrease by a reasonable amount, so that there will be a
iterates with small extended Riemannian gradient.

\begin{lemma}\label{lem:global_closeness}
Let $0 < \eps_1 \le [\minsvF^2 / (2\smoothF)] \wedge \eps_0$, $x_0 \in \cM_{\eps_1}$ and $\eta \le [\eps_1/(2 \smoothF \Gf^2)]^{1/2}$. Then $\{ x_k \}_{k \ge 0} \subseteq \cM_{\eps_1}$. 
\end{lemma}

\begin{proof}
We use proof by induction. We assume $x_k \in \cM_{\eps_1}$, i.e., we have $\| F(x_k) \|_2 \le \eps_1$. We would like to show that $\| F(x_{k+1}) \|_2 \le \eps_1$, where $x_{k+1}$ gives
\begin{equation}\label{eqn:global_iterates}
\begin{aligned}
x_{k+1}  = x_k - \eta\rgrad f(x_k) - \grad F(x_k)^\dagger F(x_k).
\end{aligned}
\end{equation}
Performing Taylor's expansion of $F(x_{k+1})$ at $x_k$, we have
\[
\begin{aligned}
&\| F(x_{k+1}) \|_2 = \| F(x_k + (x_{k+1 }- x_k)) \|_2 \\
\le ~& \| F(x_k) - \nabla F(x_k) [ \eta \rgrad f(x_k) + \grad F(x_k)^\dagger F(x_k)] \|_2  + \smoothF \| x_{k+1} - x_k \|_2^2.\\
\end{aligned}
\]
Note we have
\[
\begin{aligned}
\grad F(x_k) \rgrad f(x_k) =& \grad F(x_k) \proj_{x_k} \grad f(x_k) = 0, \\
\nabla F(x_k) \nabla F(x_k)^\dagger =& \id_n,
\end{aligned}
\]
which gives
\[
\begin{aligned}
&F(x_k) - \nabla F(x_k) [ \eta \rgrad f(x_k) + \grad F(x_k)^\dagger
F(x_k)] = F(x_k) - F(x_k) = 0.
\end{aligned}
\]
As a result, we have
\[
\begin{aligned}
&\| F(x_{k+1}) \|_2 = \smoothF \| x_{k+1} - x_k \|_2^2\\
=& \smoothF [ \eta^2 \| \rgrad f(x_k) \|_2^2 + F(x_k)^\sT \grad F(x_k)^{\dagger \sT} \grad F(x_k)^\dagger F(x_k)]\\
\le& \smoothF \eta^2 \Gf^2 + (\smoothF/\minsvF^2) \| F(x_k) \|_2^2. 
\end{aligned}
\]
Note by induction assumption, we have $\| F(x_k) \|_2 \le \eps_1$. Hence as long as $\eps_1 \le \minsvF^2 / (2\smoothF)$ and $\eta^2 \le \eps_1/(2 \smoothF \Gf^2) $, we have 
\[
\begin{aligned}
&\| F(x_{k+1}) \|_2 \le \smoothF \eta^2 \Gf^2 + (\smoothF/\minsvF^2) \| F(x_k) \|_2^2 
\le \eps_1 / 2 + \eps_1 / 2 = \eps_1. 
\end{aligned}
\]
The lemma holds by noting that the initialization of the induction holds since $x_0 \in \cM_{\eps_1}$. 
\end{proof}

\begin{lemma}\label{lem:global_convergence}
There exists constant $K < \infty$, such that for any $\eps$ satisfying $0 < \eps \le \eps_\star = [\minsvF^2 / (2\smoothF)] \wedge [\smoothF \Gf^2/(2\smoothf^2)] \wedge \eps_0$, letting $\eta = [\eps/(2 \smoothF \Gf^2)]^{1/2}$ and initializing $x_0 \in \cM_\eps$, we have for any $k \in \N$: 
\[
\min_{i \in [k]} \| \rgrad f(x_i) \|_2^2 \le K \{ [f(x_0) - \Lbfeps]/ (k \sqrt \eps) + \sqrt{\eps} \}. 
\]
\end{lemma}

\begin{proof}
Performing Taylor's expansion of $f(x_{k+1})$ at $x_k$, we have 
\[
 f(x_{k+1} ) - f(x_k) - \< \nabla f(x_k), x_{k+1} - x_k \> \le \smoothf \| x_{k+1} - x_k \|_2^2.
\]
Using Eq. (\ref{eqn:global_iterates}) gives
\[
\begin{aligned}
&f(x_{k+1}) - f(x_k) + \eta \| \rgrad f(x_k) \|_2^2 + \nabla f(x_k)^\sT \nabla F(x_k)^\dagger F(x_k) \\
\le& \smoothf[ \eta^2 \| \rgrad f(x_k)\|_2^2 + F(x_k)^\sT \nabla F(x_k)^{\dagger \sT} \nabla F(x_k)^\dagger F(x_k)]\\
\le& \smoothf[ \eta^2 \| \rgrad f(x_k)\|_2^2 + (1 / \minsvF^2) \| F(x_k) \|_2^2]. 
\end{aligned}
\]
By the choice of $\eps$ and $\eta$, we have $\eta \le 1/ (2\smoothf)$, rearranging the terms gives
\[
\begin{aligned}
&f(x_{k+1}) - f(x_k) + (\eta/2) \| \rgrad f(x_k) \|_2^2 \\
\le&  (\lipf / \minsvF) \| F(x_k) \|_2 + (\smoothf / \minsvF^2) \| F(x_k) \|_2^2\\
\le& [(\lipf / \minsvF) + (\smoothf / \minsvF^2) \eps_\star] \eps \equiv K_0 \eps. 
\end{aligned}
\]
Performing telescope summation and rearranging the terms,
\[
\begin{aligned}
&\frac{1}{k}\sum_{i=1}^k \| \rgrad f(x_i) \|_2^2 
\le  2 [f(x_0) - f(x_k)] / (k\eta) + 2 K_0 \eps / \eta
 \le  K\{ [f(x_0) - f(x_k)] / (k \sqrt \eps) + \sqrt{\eps}\},
\end{aligned}
\]
for some constant $K$. Since we know $x_k \in \cM_{\eps}$, this concludes the proof of this lemma. 
\end{proof}

\begin{figure*}[ht!]
  \centering
  \subfigure[Vary $\kappa$]{
    \includegraphics[width=0.2\textwidth]{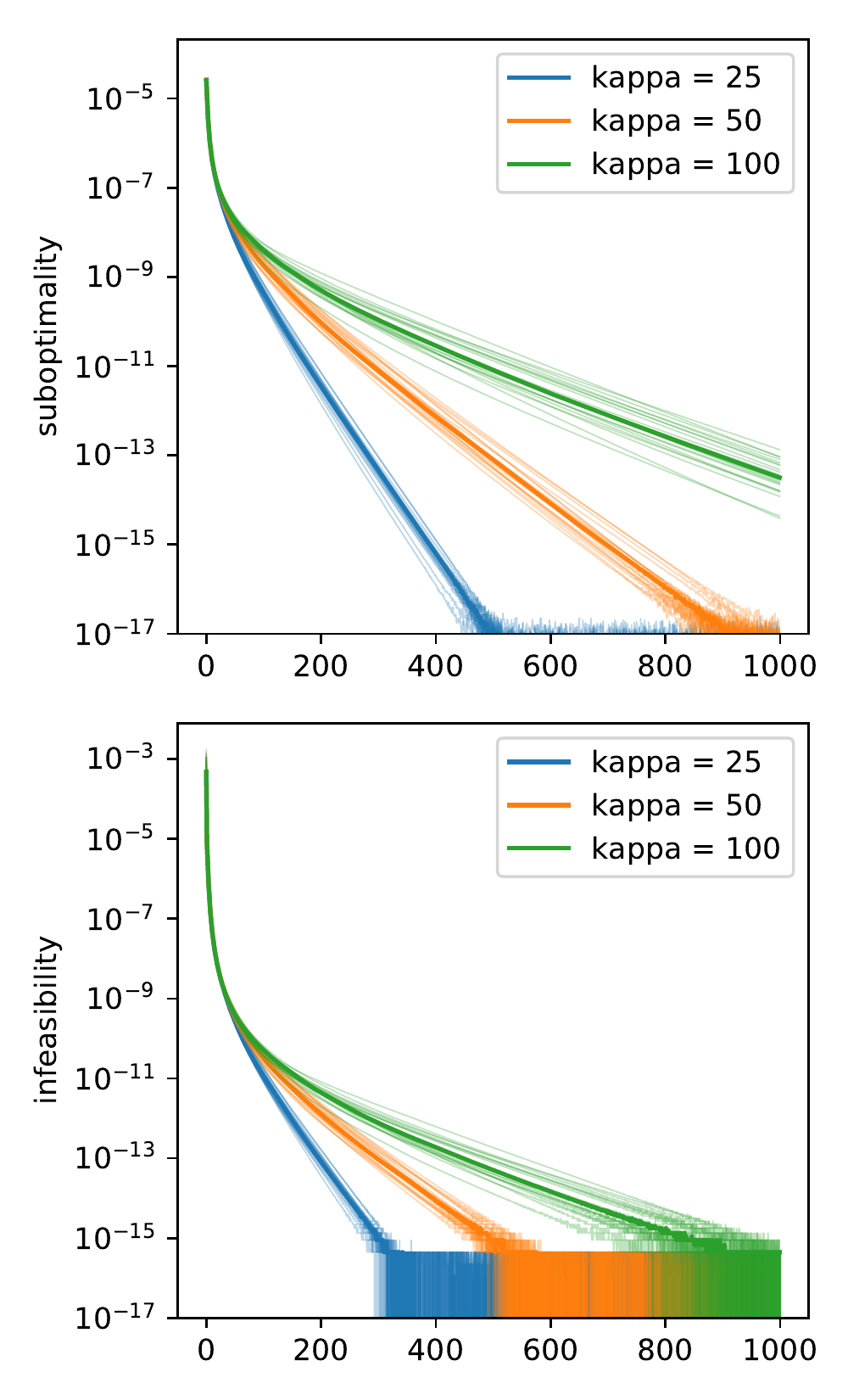}
  }\label{figure:compare-kappa}
  \subfigure[Vary $\eps$]{
    \includegraphics[width=0.2\textwidth]{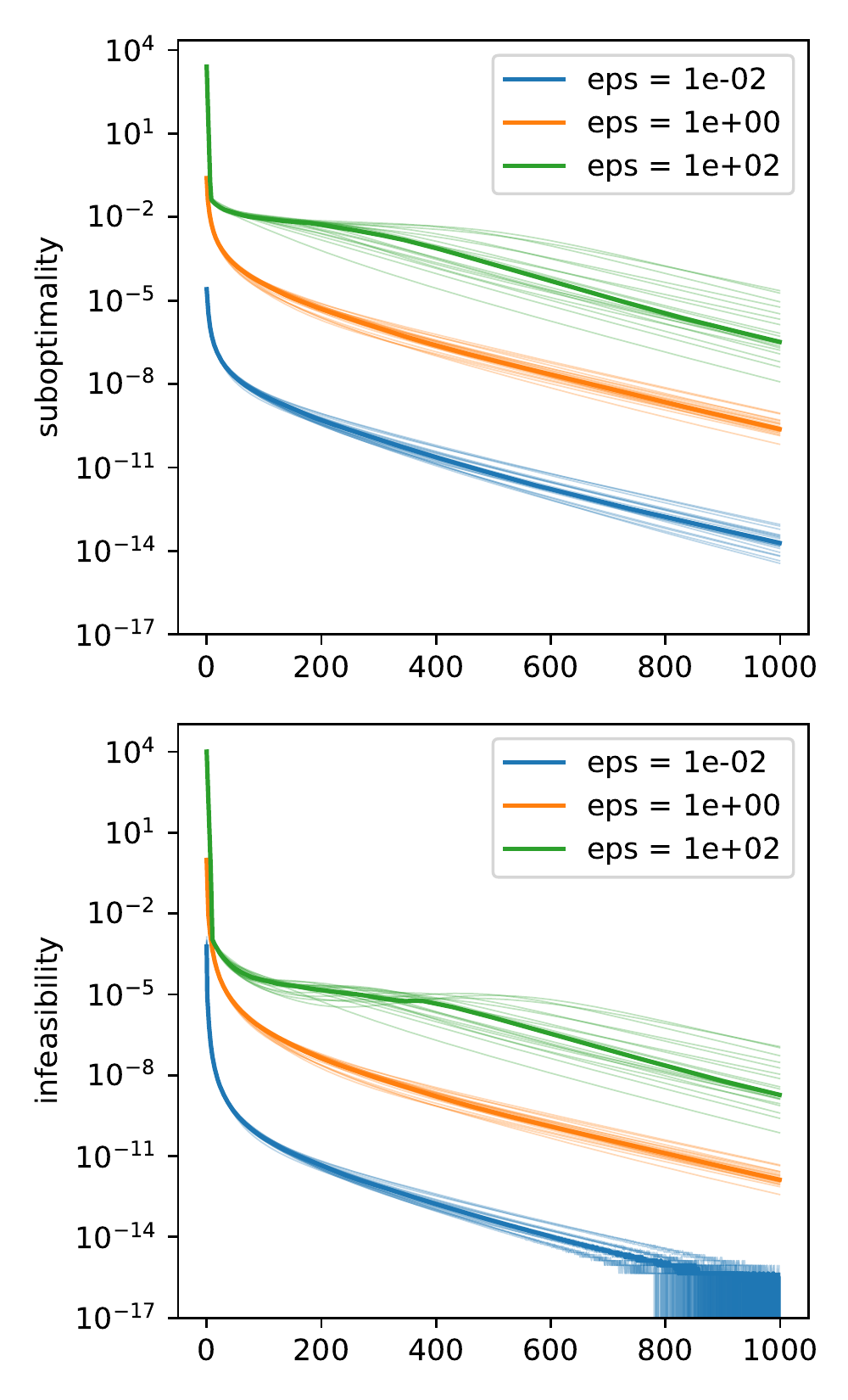}
  }\label{figure:compare-eps}
  \subfigure[Compare SQP and Riemannian]{
    \includegraphics[width=0.4\textwidth]{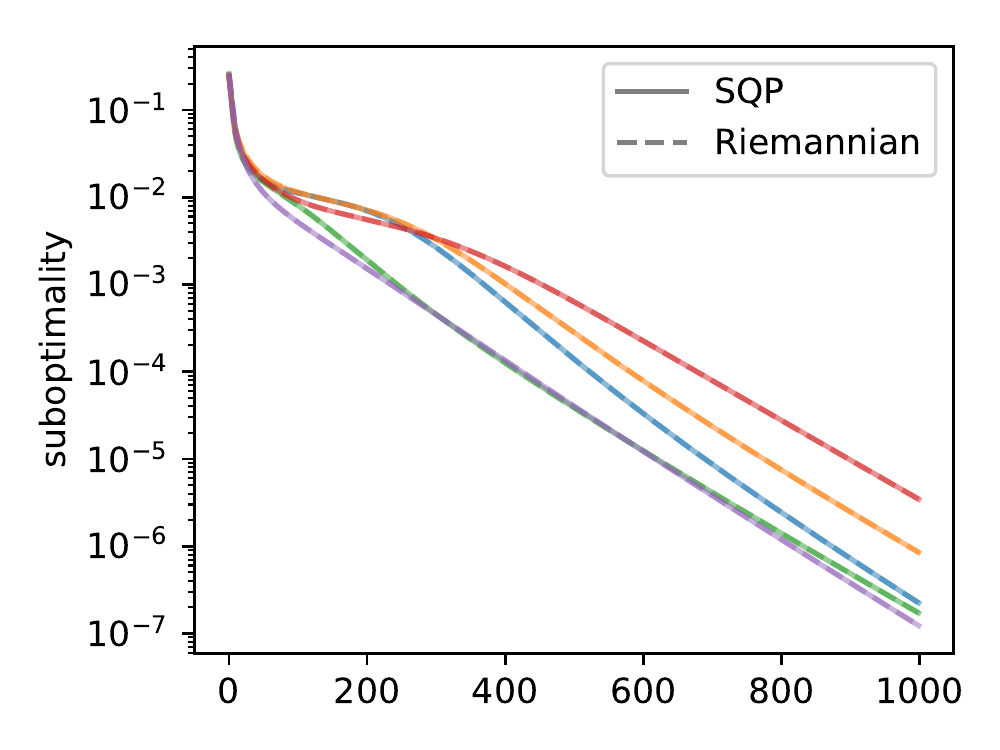}
  }\label{figure:riemannian}
  \caption{\small Convergence of SQP. (a)(b) Effect of condition
    number and initialization radius on the convergence. Thins lines
    show all the 20 instances and bold lines indicate the median
    performance. (c) Comparing SQP and Riemannian gradient descent
    over 5 instances.}
  \label{figure:figure}
\end{figure*}

\section{Example}
\label{section:example}
We provide the eigenvalue problem as a simple example illustrating our
convergence result. We emphasize that this is a simple and
well-studied problem; our goal here is only to illustrate the
connection between SQP and the Riemannian algorithms.

\begin{example}[SQP for eigenvalue problems]
  Consider the eigenvalue problem for a symmetric matrix
  $A\in\R^{n\times n}$:
  \begin{equation}
    \label{problem:eigenval}
    \begin{aligned}
      \minimize & ~~ \frac{1}{2}x^\top Ax \\
      \subjectto & ~~ \ltwo{x}^2 - 1 = 0.
    \end{aligned}
  \end{equation}
  This is an instance of problem~\eqref{problem:manopt} with
  $f(x)=(1/2)x^\top Ax$ and $F(x)=\ltwo{x}^2 - 1$. Let
  $\lambda_1<\lambda_2\le \dots\le \lambda_n$ be the eigenvalues of
  $A$, and $v_i(A)$ be the corresponding eigenvectors.

  The SQP iterate for this problem is $x\mapsto x_+=x+\Delta$, where
  $\Delta$ solves the subproblem
  \begin{align*}
    \minimize & ~~ \<Ax, \Delta\> + \frac{1}{2\eta}\ltwo{\Delta}^2 \\
    \subjectto & ~~ \ltwo{x}^2 + 2\<x, \Delta\> - 1 = 0.
  \end{align*}
  Applying~\eqref{equation:delta}, we obtain the explicit formula
  \begin{equation}
    \label{equation:sqp-eigenvalue}
    x_+ = \frac{\ltwo{x}^2 + 1}{2\ltwo{x}^2} x - \eta\Big(\id_n -
      \frac{xx^\top}{\ltwo{x}^2}\Big)Ax.
  \end{equation}
  By Theorem~\ref{theorem:convergence}, the local convergence
  rate is $1-1/\rcond$, which we now examine. At $x_\star=v_1(A)$,
  we have $\rhess f(x_\star) = A - \lambda_1 \id_n$, and the tangent
  space $\mc{T}_{x_\star}\mc{M}={\rm
    span}(v_2(A),\dots,v_n(A))$. Therefore, the condition number of
  the Riemannian Hessian is
  \begin{equation*}
    \rcond = \frac{\sup_{v\in\mc{T}_{x_\star}\mc{M},\ltwo{v}=1}\<v,
      \rhess
      f(x_\star)v\>}{\inf_{v\in\mc{T}_{x_\star}\mc{M},\ltwo{v}=1} \<v, 
      \rhess
      f(x_\star)v\>} = \frac{\lambda_n - \lambda_1}{\lambda_2 -
      \lambda_1}.
  \end{equation*}
  Hence, choosing the right stepsize $\eta$, the convergence rate of
  the SQP method is 
  \begin{equation*}
    1 - \frac{1}{\rcond} = \frac{\lambda_n - \lambda_2}{\lambda_n -
      \lambda_1},
  \end{equation*}
  matching the rate of power iteration and Riemannian gradient
  descent.

  The keen reader might find that the
  iterate~\eqref{equation:sqp-eigenvalue} converges globally to
  $x_\star$ as long as $\<x_0,x_\star\>\neq 0$. This is more
  optimistic than our
  Theorem~\ref{thm:global_closeness_and_convergence} (which only
  guarantees global convergence to stationary point). Whether such
  global convergence holds more generally would be an interesting
  direction for future study.
\end{example}

\section{Numerical experiments}
\label{sec:exp}
\paragraph{Setup}
We experiment with the SQP algorithm on random instances of the
eigenvalue problem~\eqref{problem:eigenval} with $d=1000$. Each
instance $A$ was generated randomly with a controlled condition
number $\rcond=\frac{\lambda_n - \lambda_1}{\lambda_2 - \lambda_1}$,
and the stepsize was chosen as $\eta=\frac{1}{2(\lambda_n-\lambda_1)}$
to optimize for the local linear rate. The initialization $x_0$ is
sampled randomly around the solution $x_\star$ with average distance
$\eps$ (recall $\ltwo{x_\star}=1$).

We run the following three sets of comparisons and plot the results in
Figure~\ref{figure:figure}.
\begin{enumerate}[(a)]
\item Test the effect of $\rcond$ on the local linear rate, with
  $\rcond\in\set{25, 50, 100}$ and fixed $\eps=0.01$.
\item Test the effect of initialization radius $\eps$ (i.e. localness)
  on the convergence, with $\eps\in\set{0.01, 1, 100}$ and fixed
  $\kappa=100$.
\item Test whether SQP is close to Riemannian gradient descent when
  initialized at a same feasible start $x_0\in\mc{M}$.
\end{enumerate}

\paragraph{Results} In experiment (a), we see indeed that the local
linear rate of SQP scales as $1-C/\rcond$: doubling the condition
number will double the number of iterations required for halving the
sub-optimality. Experiment (b) shows that the linear convergence is
indeed more robust locally than globally; when
initialized very far away ($\eps=100$), linear convergence with the
same rate happens on most of the instances after a while, but there
does exist bad instances on which the convergence is slow. This
corroborates our theory that the global convergence of SQP is more
sensitive to the stepsize choice as the SQP is not guaranteed to
approach the manifold when initialized far away. Experiment (c)
verifies our intuition that SQP is approximately equal to Riemannian
gradient descent: with a feasible start, their iterates stay almost
exactly the same.

\section{Conclusion}
We established local and global convergence of a cheap SQP algorithm
building on intuitions from Riemannian optimization. Potential future
directions include generalizing our global result to ``far from the
manifold'', as well as identifying problem structures under which we
obtain global convergence to local minimum.

\section*{Acknowledgement}
We thank Nicolas Boumal and John Duchi for a number of helpful
discussions. YB was partially supported by John Duchi's National
Science Foundation award CAREER-1553086. SM was supported by an Office
of Technology Licensing Stanford Graduate Fellowship.

\bibliographystyle{amsalpha}
\bibliography{main}

\newpage

\appendix

\section{Proof of technical results}
\subsection{Some tools}
\begin{lemma}[\cite{MengZh10}]
  \label{lemma:pseudoinverse-bound}
  For $x_1,x_2 \in \ball(x_\star, \delta)$, we have
  \begin{equation*}
    \dl \nabla F(x_1)^\dagger - \nabla F(x_2)^\dagger \dl_{\op} \le
    \smoothD \dl x_1-x_2 \dl_2.
  \end{equation*}
  where $\smoothD = 2\smoothF/\minsvF^2$.
\end{lemma}

\begin{lemma}
  \label{lemma:riemannian-gradient-lipschitz}
  The extended Riemannian gradient $\rgrad f(x)$ is
  $\smoothr$-Lipschitz in $\ball(x_\star, \delta)$, where
  $\smoothr=\smoothP\lipf+\smoothf$. 
\end{lemma}
\begin{proof}
  We have
  \[
  \begin{aligned}
  \dl \rgrad f(x_1) - \rgrad f(x_2) \dl_2 = & \dl \proj_{x_1} \nabla f(x_1) - \proj_{x_2} \nabla f(x_2) \dl_2 \\
  \le& \dl (\proj_{x_1} - \proj_{x_2}) \nabla f(x_1) \dl_2 + \dl \proj_{x_2} (\nabla f(x_1) - \nabla f(x_2)) \dl_2\\
  \le& (\smoothP \lipf + \smoothf)\dl x_1 - x_2 \dl_2.
  \end{aligned}
  \]
\end{proof}

\subsection{Proof of Lemma~\ref{lemma:riemannian-taylor-expansion}}
\label{appendix:proof-riemannian-taylor-expansion}

For any $x \in \ball(x_\star, \eps_0)$, denote $x_p = \proj_{x_\star} (x - x_\star) + x_\star$. We prove the following equation first
\begin{align}\label{eqn:local_riemannian_expansion_tangent}
\rgrad f(x_p) = \rhess f(x_\star) (x_p - x_\star) + r(x_p),
\end{align}
where $r(x_p) \le \cra \| x_p - x_\star \|_2^2$. 

First, we show that $\rgrad f(x_p)$ can be well approximated by $\proj_{x_\star}\rgrad f(x_p)$. Denoting $r_0 = \proj_{x_\star} \rgrad f(x_p) - \rgrad f(x_p)$, by Lemma~\ref{lemma:projection-perturbation-bound} and~\ref{lemma:riemannian-gradient-lipschitz}, we have
\begin{align*}
\dl r_0 \dl_2 = & \dl\proj_{x_\star}^\perp \rgrad f(x_p) \dl_2=
  \dl \proj_{x_\star}^\perp \proj_{x_p} \rgrad f(x_p) \dl_2 \\
 \le&
\dl \proj_{x_\star}^\perp \proj_{x_p} \dl_{\op} \cdot \dl \rgrad f(x_p) \dl_2
\le \smoothP \dl x_p-x_\star \dl_2 \dl \rgrad f(x_p)\dl_2 \le
  \smoothP\smoothr \dl x_p-x_\star \dl_2^2.
\end{align*}

Then, for any $u \in \R^n$ with $\dl u \dl_2 = 1$, we have
\[
\begin{aligned}
&\< u, \proj_{x_\star} \rgrad f(x_p)  \> \\
=&\< u, \proj_{x_\star}\proj_{x_p}\{\nabla f(x_\star) + \nabla^2 f(x_\star) (x_p - x_\star) + 1/2 \cdot \nabla^3 f(\tilde x_p) [\,\cdot\,,(x_p - x_\star)^{\otimes 2}]\} \>\\
=& \< u, \proj_{x_\star}\proj_{x_p}[\nabla f(x_\star) + \nabla^2 f(x_\star) (x_p - x_\star)]\> + r_1,\\
\end{aligned}
\]
where
\[
\vert r_1 \vert =1/2 \cdot \vert \< u, \proj_{x_\star} \proj_{x_p} \nabla^3 f(\tilde x_p) [\,\cdot\,,(x_p - x_\star)^{\otimes 2}] \> \le 1/2 \cdot \rhof \dl x_p - x_\star \dl_2^2. 
\]
Then we have
\[
\begin{aligned}
&\< u, \proj_{x_\star} \rgrad f(x_p)  \> \\
=& \< u, \proj_{x_\star}\proj_{x_p}[\nabla f(x_\star) + \nabla^2 f(x_\star) (x_p - x_\star)]\> + r_1,\\
=&\< u, \proj_{x_\star}\proj_{x_p} \nabla f(x_\star)\> +  \< u, \proj_{x_\star} \nabla^2 f(x_\star) (x_p - x_\star) \>  + \< u, \proj_{x_\star}( \proj_{x_p} - \proj_{x_\star} )\nabla^2 f(x_\star) (x_p - x_\star) \> + r_1\\
=& \underbrace{\< u, \proj_{x_\star}\proj_{x_p} \nabla f(x_\star)\>}_{\rm I} +  \< u, \proj_{x_\star} \nabla^2 f(x_\star) (x_p - x_\star) \> + r_1 + r_2,
\end{aligned}
\]
where 
\[
\vert r_2 \vert = \vert \< u, \proj_{x_\star}( \proj_{x_p} - \proj_{x_\star} )\nabla^2 f(x_\star) (x_p - x_\star) \> \vert \le \smoothP \smoothf \dl x_p - x_\star \dl_2^2. 
\]

Then we look at the term $\rm I$. We have
\[
\begin{aligned}
\rm I =& - \< u, \proj_{x_\star}\nabla F(x_p)^\sT  \nabla F(x_p)^{\dagger \sT} \nabla f(x_\star)\> \\
=& - \< u, \proj_{x_\star}[\nabla F(x_p) - \nabla F(x_\star)]^\sT \nabla F(x_p)^{\dagger \sT} \nabla f(x_\star)\> \\
=& - \< u, \proj_{x_\star}\{\nabla^2 F(x_\star)[\, \cdot \,, \, \cdot\, , x_p - x_\star] + 1/2 \cdot \nabla^3 F(\tilde x_p) [\, \cdot \,, \, \cdot\, ,(x_p - x_\star)^{\otimes 2}]\}^\sT \nabla F(x_p)^{\dagger \sT} \nabla f(x_\star)\> \\
=& - \sum_{i=1}^m [ \nabla F(x_p)^{\dagger \sT} \nabla f(x_\star)]_i \nabla^2 F_i(x_\star)[x_p - x_\star, \proj_{x_\star} u] + r_3,
\end{aligned}
\]
where 
\[
\begin{aligned}
\vert r_3 \vert =& 1/2 \cdot \vert \sum_{i=1}^m [\nabla F(x_p)^{\dagger \sT} \nabla f(x_\star)]_i \nabla^3 F_i(\tilde x_p)[(x_p - x_\star)^{\otimes 2}, \proj_{x_\star} u] \vert \\
\le & \rhoF \lipf  / (2\minsvF)  \cdot \dl x_p - x_\star \dl_2^2.
\end{aligned}
\]
We further have 
\[
\begin{aligned}
\rm I =& - \sum_{i=1}^m [ \nabla F(x_p)^{\dagger \sT} \nabla f(x_\star)]_i \nabla^2 F_i(x_\star)[x_p - x_\star, \proj_{x_\star} u] + r_3\\
=& - \sum_{i=1}^m [ \nabla F(x_\star)^{\dagger \sT} \nabla f(x_\star)]_i \nabla^2 F_i(x_\star)[x_p - x_\star, \proj_{x_\star} u] + r_3 + r_4,
\end{aligned}
\]
where 
\[
\vert r_4 \vert = \sum_{i=1}^m \{ [\nabla F(x_p) - \nabla F(x_\star)]^{\dagger \sT} \nabla f(x_\star)\}_i \nabla^2 F_i(x_\star)[x_p - x_\star, P_{x_\star} u]\le \smoothD\lipf\smoothF  \cdot \dl x_p - x_\star \dl_2^2.
\]

Above all, combining all the terms, we have
\[
\begin{aligned}
&\vert \< u, \rgrad f(x_p) - \rhess f(x_p) (x_p - x_\star) \> \vert \le \dl r_0 \dl_2 +  \vert r_1 + r_2 + r_3 + r_4 \vert \\
 \le& (\smoothP\smoothr + 1/2 \cdot \rhof + \smoothP \smoothf +  \rhoF \lipf  / (2\minsvF) + \smoothD\lipf\smoothF) \dl x_p - x_\star \dl_2^2. 
\end{aligned}
\]
Taking $\cra = \smoothP\smoothr + 1/2 \cdot \rhof + \smoothP \smoothf + \rhoF \lipf  / (2\minsvF)+ \smoothD\lipf\smoothF$ we get Eq. (\ref{eqn:local_riemannian_expansion_tangent}).

Then by Lemma \ref{lemma:riemannian-gradient-lipschitz}, we have 
\begin{align}
\| \rgrad f(x) - \rgrad f(x_p) \| \le \crb \| x - x_p \|_2. 
\end{align}
This proves the lemma. 

\subsection{Proof of Corollary~\ref{corollary:rgrad-expansions}}
\label{appendix:proof-rgrad-expansions}
Let $\eps_0$ be given by
Lemma~\ref{lemma:riemannian-taylor-expansion}, $\eps\le \eps_0$, and 
$x\in\ball(x,\eps)$. We have
\begin{equation}
  \rgrad f(x) = \rhess f(x_\star)(x-x_\star) + r(x),~~\| r(x)\|_2 \le
  \cra\| x-x_\star\|_2^2 + \crb\| \proj_{x_\star}^\perp(x-x_\star)\|_2.
\end{equation}
Therefore we obtain the expansion
\begin{equation*}
  \<\rgrad f(x), x-x_\star\> = \<x-x_\star, \rhess
  f(x_\star)(x-x_\star)\> + R_1,
\end{equation*}
where $R_1$ is bounded as
\begin{equation}
  \begin{aligned}
    & \quad |R_1| = |\<r(x), x-x_\star\>| \le \cra \| x-x_\star\|_2^3 +
    \crb \| x-x_\star\|_2 \| \proj_{x_\star}^\perp(x-x_\star) \|_2 \\
    &  \le \eps\cra \| x-x_\star\|_2^2 +
    \eps\crb\| \proj_{x_\star}^\perp(x-x_\star)\|_2. 
  \end{aligned}
\end{equation}
Similarly, we have the expansion
\begin{equation*}
  \| \rgrad f(x_\star)\|_2^2 = \<x-x_\star, (\rhess
    f(x_\star) )^2(x-x_\star)\> + R_2,
\end{equation*}
where $R_2$ is bounded as (letting
$H=\lambda_{\max}(\rhess f(x_\star))$ for convenience)
\begin{equation}
  \begin{aligned}
    & \quad |R_2| \le 2|\<\rhess f(x_\star)(x-x_\star), r(x)\>| +
    \| r(x)\|_2^2 \\
    & \le 2H\big(\cra \| x-x_\star\|_2^3 +
      \crb\| x-x_\star\|_2 \| \proj_{x_\star}^\perp(x-x_\star)\|_2 \big)
    \\
    & \quad + \big(\cra^2\| x-x_\star\|_2^4 + 
      2\cra\crb\| x-x_\star\|_2^2\| \proj_{x_\star}^\perp(x-x_\star) 
      \|_2 + \crb^2\| \proj_{x_\star}^\perp(x-x_\star)\|_2^2\big) \\
    & \le (2H\cra\eps+\cra^2\eps^2)\| x-x_\star\|_2^2 +
    (2H\crb\eps +
    2\cra\crb\eps^2)\| \proj_{x_\star}^\perp(x-x_\star)\|_2 +
    \crb^2\eps\| \proj_{x_\star}^\perp(x-x_\star)\|_2 \\
    & \le
    \eps\underbrace{(2H\cra+\cra^2\eps_0)}_{\crat}\| x-x_\star\|_2^2
    + \eps\underbrace{(2H\crb+2\cra\crb\eps_0+\crb^2\eps_0)}_{\crbt}
    \| \proj_{x_\star}^\perp(x-x_\star)\|_2 \\
    & = \eps\big(\crat\| x-x_\star\|_2^2 +
      \crbt\| \proj_{x_\star}^\perp(x-x_\star)\|_2 \big).
  \end{aligned}
\end{equation}
Overloading the constants, we get
\begin{align*}
  & \quad \max\set{|R_1|, |R_2|} \le
    \eps\big(\cra\| x-x_\star\|_2^2 + 
    \crb\| \proj_{x_\star}^\perp(x-x_\star)\|_2\big) \\
  & = \eps\big(\cra\| \proj_{x_\star}(x-x_\star)\|_2^2 +
    \cra\|\proj_{x_\star}^\perp(x-x_\star)\|_2^2 +
    \crb\| \proj_{x_\star}^\perp(x-x_\star)\|_2 \big) \\
  & \le \eps\big(\cra\| x-x_\star\|_2^2 +
    \underbrace{(\crb+\cra\eps_0)}_{\crb}
    \|\proj_{x_\star}^\perp(x-x_\star)\|_2\big).
\end{align*}

\end{document}